\documentclass[11pt]{amsart}
\usepackage{amssymb}
\theoremstyle{plain}
\newtheorem{theorem}{Theorem}
\newtheorem{corollary}[theorem]{Corollary}
\newtheorem{lemma}[theorem]{Lemma}

\theoremstyle{remark}

\begin{document}

\title[Subgroups generated by generic automorphisms]{Closed subgroups generated by\\ generic measure automorphisms}

\author{S\l awomir Solecki}

%\date{October 2012}

\address{Department of Mathematics\\
1409 W. Green St.\\
University of Illinois\\
Urbana, IL 61801, USA}

\email{ssolecki@math.uiuc.edu}

\thanks{Research supported by NSF grant DMS-1001623.}


\begin{abstract}
We prove that for a generic measure preserving transformation $T$, the
closed group generated by $T$ is a continuous homomorphic image of a closed
linear
subspace of $L_0(\lambda, {\mathbb R})$, where $\lambda$ is Lebesgue measure, and that the
closed group generated by $T$ contains an increasing sequence of finite
dimensional toruses whose union is dense.
\end{abstract}


\maketitle

\section{Results}\label{S:res}
For a Borel atomless probability measure $\mu$ on a standard Borel space, by
${\rm Aut}(\mu)$
we denote the group, taken with composition as the group operation, of all measure classes
of invertible $\mu$-preserving transformations. The topology on it is the weakest topology
making the functions
\[
{\rm Aut}(\mu)\ni T \to \mu(TA\triangle A) \in {\mathbb R}
\]
continuous for each Borel set $A$. In this fashion ${\rm Aut}(\mu)$ becomes
a Polish group. For $T\in {\rm Aut}(\mu)$, $\langle T\rangle_c$ will denote the {\em closed} group
generated by $T$, that is, we set
\[
\langle T \rangle_c = \hbox{closure}(\{ T^n\colon n\in {\mathbb Z}\}).
\]
We study the structure of this group for a generic $T\in {\rm Aut}(\mu)$, that is, we are interested in
properties of this group that are exhibited on a comeager subset of $T\in {\rm Aut}(\mu)$. Our
results strengthen theorems of Ageev \cite{Ag} and de la Rue and de Sam Lazaro \cite{RS}.
They are also related to a question of Glasner and Weiss \cite{GW}. Below we say a bit more about it.
The methods developed in \cite{Ag} and in \cite{RS} play an important
role in our proofs. Additionally, we rely on the results from \cite{ACS} and \cite{Ki2}.

To state our main result, we need to recall some definitions.
Let $\lambda$ be Lebesgue measure on $[0,1]$. Let $L_0(\lambda,
{\mathbb R})$ be the linear space, with pointwise linear
operations, of all measure classes of $\lambda$-measurable
functions from $[0,1]$ to $\mathbb R$. We equip this space with the
topology of convergence in measure. This topology is Polish and
makes the linear operations continuous. In some situations, we consider $L_0(\lambda, {\mathbb R})$,
as a Polish group with vector addition as the group operation. Similarly,
let $L_0(\lambda, {\mathbb T})$ be the group, with pointwise
multiplication, of all measure classes of $\lambda$-measurable
functions from $[0,1]$ to the circle group ${\mathbb T}$. Again, we take
it with the topology of convergence in measure. This topology
makes $L_0(\lambda, {\mathbb T})$ into a Polish group. There is a
canonical continuous homomorphism
\[
\exp \colon L_0(\lambda, {\mathbb R}) \to L_0(\lambda, {\mathbb T})
\]
given by
\[
f\to e^{i f}.
\]

We will prove the following theorem.

\begin{theorem}\label{T:expo}
Let $\mu$ be a Borel atomless probability measure on a standard Borel
space. There exists a comeager subset $H$ of ${\rm Aut}(\mu)$ such
that for each $T\in H$ there exists a closed linear subspace $L$
of $L_0(\lambda, {\mathbb R})$ with the property that $\langle
T\rangle_c$ is topologically isomorphic to the subgroup $\exp(L)$ of
$L_0(\lambda, {\mathbb T})$.
\end{theorem}

In the corollary below we collect a couple of consequences of Theorem~\ref{T:expo}.
Define the topological group
\[
{\mathbb T}_\infty
\]
by letting it be the direct limit of the sequence ${\mathbb T}^n$, $n\in {\mathbb N}$,
where ${\mathbb T}^n$ is identified with the closed subgroup of ${\mathbb T}^{n+1}$ consisting of elements whose last coordinate is equal to $1$.
So open sets in ${\mathbb T}_\infty$ are precisely those sets whose intersection with each ${\mathbb T}^n$ is open.
Clearly ${\mathbb T}_\infty$ is not a Polish group.

\begin{corollary}\label{C:pro}
There is a comeager set $H\subseteq {\rm Aut}(\mu)$ such that for each $T\in H$
\begin{enumerate}
\item[(i)]  $\langle T\rangle_c$ is a continuous homomorphic image of a closed linear subspace of $L_0(\lambda, {\mathbb R})$;

\item[(ii)] there is a continuous
embedding of ${\mathbb T}_\infty$ into $\langle T\rangle_c$ whose image is dense in $\langle T\rangle_c$.
\end{enumerate}
\end{corollary}

In Corollary~\ref{C:pro}, point (i) strengthens de la Rue and de Sam
Lazaro's result from \cite{RS} that a generic transformation lies on
a one-parameter subgroup of ${\rm Aut}(\mu)$. Point (ii)
strengthens the result of Ageev \cite{Ag}
that for a generic $T\in {\rm Aut}(\mu)$ the group $\langle T\rangle_c$ contains an
arbitrary finite abelian group.

Glasner and Weiss in \cite{GW} asked if $\langle T\rangle_c$ is a L{\'e}vy group for a generic
$T\in {\rm Aut}(\mu)$. Actually, it is not even ruled out at this point that $\langle T\rangle_c$ is
isomorphic to $L_0(\lambda, {\mathbb T})$. Theorem~\ref{T:expo} and Corollary~\ref{C:pro}
provide some evidence that such results may be true. More evidence was subsequently found by Melleray and Tsankov in
\cite{MT}, where they showed that the group $\langle T\rangle_c$ is extremely amenable for a generic $T\in {\rm Aut}(\mu)$. One may
hope that extreme amenability and an increasing sequence of toruses with dense union, as in Corollary~\ref{C:pro}(ii), will imply
being L{\'e}vy, but such an implication is false in general as shown in \cite{FS}.
Additionally, it is proved in \cite{MT} that the closed subgroup generated by a generic unitary transformation of a separable infinite dimensional
Hilbert space is isomorphic to $L_0(\lambda, {\mathbb T})$.

Theorem~\ref{T:expo} can be combined with the methods of \cite{Ag} to obtain also
the result of Stepin and Eremenko \cite{SE} that $\langle T\rangle_c$ contains the group ${\mathbb T}^{\mathbb N}$
for a generic $T\in {\rm Aut}(\mu)$. (This theorem, which strengthens an earlier result, announced in \cite{Ag2},
that $\langle T\rangle_c$ contains an arbitrary countable abelian group, is stated in the abstract of \cite{SE}. It is not proved explicitly in that
paper, but Aaron Hill points out that it can be deduced from \cite[Theorem 1.3]{SE}.)
Furthermore, using the approach of the present paper and the part of the work of Melleray and Tsankov \cite{MT} that
pertains to the isometry group of the Urysohn space, one could show that a result analogous to
Corollary~\ref{C:pro}(i) holds for this group (the closed group generated by a generic isometry of the
Urysohn space is a continuous homomorphic image of a separable F-space) provided one could show that a generic isometry of the
Urysohn space lies on a one-parameter group. However, the question whether this is true remains open.

\section{Proofs}
We will need to following lemma that resembles the Kuratowski--Ulam theorem.

\begin{lemma}\label{L:auto}
Let $A\subseteq {\rm Aut}(\mu)$ have the Baire property. Then $A$
is meager (comeager, respectively) if and only if there is a
comeager set of $T\in {\rm Aut}(\mu)$ such that $A\cap \langle T\rangle_c$ is meager (comeager, respectively) in
$\langle T\rangle_c$.
\end{lemma}

\begin{proof} We will need the following result of King \cite[Theorem (23a), p.538]{Ki2}:

\noindent For $q\in {\mathbb Z}\setminus \{ 0\}$, the function
$f_q: {\rm Aut}(\mu)\to {\rm Aut}(\mu)$ given by $f_q(T) = T^q$
has the property that the image of each open non-empty subset of
${\rm Aut}(\mu)$ is somewhere dense.

\noindent (This theorem is stated in \cite{Ki2} for $q>1$, but it is clear for $q=1$, and
the case $q< 0$ follows from the case $q>0$ since $T\to T^{-1}$ is a homeomorphism
of ${\rm Aut}(\mu)$.)

The property from the conclusion of this result implies that the
images under $f_q$, $q\in {\mathbb Z}\setminus \{ 0\}$, of
non-meager sets are non-meager. It follows from this statement
that preimages under $f_q$ of comeager sets are comeager. Let $A$
be comeager. Let $B\subseteq A$ be a $G_\delta$ set dense in ${\rm
Aut}(\mu)$. Set
$$
B_\infty  = \bigcap_{q\in {\mathbb Z}\setminus \{ 0\}}
f_q^{-1}(B).
$$
Observe that, by what was said above, $B_\infty$ is comeager. We also
have that for each $T\in B_\infty$
$$
\{ T^q: q\in {\mathbb Z}\setminus \{ 0\}\}\subseteq B.
$$
Note that the set of all $T\in {\rm Aut}(\mu)$ for which $\{ T^q: q\in {\mathbb Z}\}$ is not discrete is
a $G_\delta$ and it is dense in ${\rm Aut}(\mu)$ since it contains all transformations isomorphic to
irrational rotations of the circle.
Let $C$ consist of those $T\in B_\infty$ for which the set
$\{ T^q: q\in {\mathbb Z}\}$ is not discrete. Then $C$ is still
comeager. It follows, that for each $T\in C$, $B\cap \langle T\rangle_c$ is dense in $\langle T\rangle_c$.
Since $B$ is a $G_\delta$, we have that $B\cap \langle T\rangle_c$, and therefore also
$A\cap \langle T\rangle_c$, is comeager in $\langle T\rangle_c$, as required.

To complete the proof of the lemma it suffices to show that if $A$
is non-meager, then there exists a non-meager set of $T\in {\rm
Aut}(\mu)$ such that $A\cap \langle T\rangle_c$ is non-meager. Since $A$ is
non-meager, it is comeager in some non-empty open set. This open
sets contains a transformation with dense conjugacy class by \cite[Theorem 1]{Ha}. Thus,
there exist $S_n\in {\rm Aut}(\mu)$, $n\in {\mathbb N}$, such that
$\bigcup_n S_n A S_n^{-1}$ is comeager. So, there is a comeager
set of $T$ such that $\bigcup_n S_n A S_n^{-1}$ is comeager in
$\langle T\rangle_c$. It follows that there exits $n_0$ and a non-meager set of
$T$ such that $S_{n_0}AS_{n_0}^{-1}$ is non-meager in $\langle T\rangle_c$. For
each such $T$, $A$ is non-meager in
$$
S_{n_0}^{-1}\langle T\rangle_c S_{n_0}= \langle S_{n_0}^{-1}TS_{n_0}\rangle_c.
$$
Thus, there exists a non-meager set of $T'$ with $A$ non-meager in
$\langle T'\rangle_c$.
\end{proof}

Let us point out that Lemma~\ref{L:auto} can also be derived using the methods of \cite[Section
8]{MT}; one would use Lemma 8.7, Theorem A.5, and the proof of Theorem 8.13 of that paper. These methods were
invented after the authors of \cite{MT} became aware of the statements of the above lemma and of Theorem~\ref{T:expo}.

We will also need the property stated in the lemma below of the map $\exp\colon L_0(\lambda, {\mathbb R})\to
L_0(\lambda, {\mathbb T})$ defined in Section~\ref{S:res}. This property makes the map into
an analogue of the exponential map on Lie groups. The lemma is folklore and follows easily from
Mackey's cocycle theorem. We include its proof here for the sake of completeness.

\begin{lemma}\label{L:lpa}
If $X\colon {\mathbb R}\to L_0(\lambda, {\mathbb T})$ is a
one-parameter subgroup of the group $L_0(\lambda, {\mathbb T})$, then there exists a
unique element $g\in L_0(\lambda, {\mathbb R})$ such that for each
$t\in {\mathbb R}$
\[
X(t) = \exp(tg).
\]
\end{lemma}

\begin{proof} The one-parameter group $X$, gives rise to a Borel function
$\phi\colon [0,1]\times {\mathbb R}\to {\mathbb T}$ such that for each $t\in {\mathbb R}$, the function $\phi(\cdot, t)$ is in the $\lambda$-class $X(t)\in L_0(\lambda, {\mathbb T})$
and with the property that for almost all with respect to the product of Lebesgue measures $(x,r,s)\in [0,1]\times {\mathbb R}^2$
we have
\begin{equation}\label{E:cocy}
\phi(x,r+s) = \phi(x,r)\phi(x,s).
\end{equation}
By Mackey's theorem \cite[Theorem IV.9]{Fa} by changing $\phi$ on a set of measure zero in $[0,1]\times {\mathbb R}^2$, we can assume
that \eqref{E:cocy} holds for all $x$, $r$, and $s$ and that $\phi$ is still  Borel. Now for each $x\in [0,1]$, $\phi(x, \cdot)$ is a continuous
homomorphism from $\mathbb R$ to ${\mathbb T}$. Thus, there exists a unique real number $g(x)\in {\mathbb R}$ such that for all $t\in {\mathbb R}$,
$\phi(x,t) = e^{itg(x)}$.
It is easy to check that so defined function $g\colon [0,1]\to {\mathbb R}$ is Borel, that it fulfills the conclusion of the lemma, and that it is unique as
an element of $L_0(\lambda, {\mathbb R})$.
\end{proof}

\begin{proof}[Proof of Theorem~\ref{T:expo}]
We will specify three properties that we will show to hold of a generic $T\in {\rm Aut}(\mu)$ and
then prove that the conjunction of these properties implies the property of $T$ in the conclusion of the theorem.

Recall that an {\em F-space} is a topological vector space whose topology
is given by a complete metric. For example,
$L_0(\lambda, {\mathbb R})$ is a separable F-space.

We fix some notation in the broader context of Polish groups. Let
$G$ be a Polish group. Let ${\mathcal L}(G)$ be the space of all one
parameter subgroups of $G$, that is, all continuous homomorphisms
from $\mathbb R$ to $G$. There are three elements of the structure
of ${\mathcal L}(G)$ we want to point out. All the assertions about
them made below are easy to prove and we leave it to the reader to
supply these proofs.

1. ${\mathcal L}(G)$ is a Polish space when equipped with the
topology it inherits from the space of all continuous functions from
$\mathbb R$ to $G$ taken with the compact-open topology.

2. The function
\[
{\mathbb R}\times {\mathcal L}(G)\ni (r, X)\to rX \in {\mathcal L}(G)
\]
given by
\[
(rX)(t) = X(rt)
\]
defines a continuous with respect to the topology from point 1 action of
$\mathbb R$ on ${\mathcal L}(G)$.

3. The function $e: {\mathcal L}(G)\to G$ given by
\[
e(X)= X(1)
\]
is continuous.

For $T\in {\rm Aut}(\mu)$, let
\[
L_T=\{ X\in {\mathcal L}({\rm Aut}(\mu)): {\rm image}(X)\subseteq \langle T\rangle_c\}.
\]
By points 2 and 3, since $\langle T\rangle_c$ is closed, $L_T$ is a closed subset of ${\mathcal
L}({\rm Aut}(\mu))$. By the fact that $\langle T\rangle_c$ is abelian and by point 1, $L_T$ is a separable F-space
with multiplication by reals given by the action of $\mathbb R$ on
${\mathcal L}({\rm Aut}(\mu))$ from point 2 and addition given by
\[
(X+Y)(t) = X(t)Y(t).
\]

We will need the following result of de la Rue and de Sam
Lazaro \cite{RS}, which we state here in our terminology:

\noindent the image of $e: {\mathcal L}({\rm Aut}(\mu)) \to {\rm
Aut}(\mu)$ is comeager.

\noindent Now, it follows from Lemma~\ref{L:auto} that there exists a comeager
set of $T$ such that ${\rm image}(e)$ is comeager in $\langle T\rangle_c$. From this point on, our generic
transformation $T$ is assumed to have this property, that is, ${\rm
image}(e)$ is comeager in $\langle T\rangle_c$.

For an element $g$ of a group $G$, let $C(g)$ stand for the
centralizer of $g$, that is, $C(g) = \{ h\in G: hg=gh\}$. We will need
the following result of Akcoglu, Chacon, and Schwartzbauer \cite[Theorem 3.3]{ACS}, for another proof of which the reader may consult \cite{MT}:

\noindent there is a comeager set of $S\in {\rm Aut}(\mu)$ such
that $\langle S\rangle_c = C(S)$.

\noindent (In \cite{ACS}, it is actually proved that $\langle S\rangle_c = C(S)$ holds for all
transformations $S$ with strong approximation by partitions. It is not difficult to see, and it follows directly from \cite[Theorem 1.1]{KS},
that the set of all transformations with strong approximation by partitions is comeager in ${\rm Aut}(\mu)$.)

Now, it follows from Lemma~\ref{L:auto} that there is a comeager set
of $T$ such that
\[
A_T = \{ S\in \langle T\rangle_c\colon \langle S\rangle_c = C(S)\}
\]
is comeager in $\langle T\rangle_c$. From this point on we assume that our $T$ has
this property.

We aim to show that $e\upharpoonright L_T$ is onto $\langle T\rangle_c$.
If $S\in A_T$ and
$X\in {\mathcal L}({\rm Aut}(\mu))$ is such that $e(X) = S$, then
\[
{\rm image}(X)\subseteq C(S)=\langle S\rangle_c\subseteq
\langle T\rangle_c.
\]
It follows that $X\in L_T$. Thus, we get
\[
e(L_T)\supseteq A_T\cap {\rm image}(e).
\]
The set on the right-hand side is comeager in $\langle T\rangle_c$, hence $e(L_T)$ is
comeager in $\langle T\rangle_c$. Since $e\upharpoonright L_T$ is a group homomorphism,
considering $L_T$ as a group with vector addition, we see that $e(L_T)$ is a comeager
subgroup of the Polish group $\langle T\rangle_c$. It follows that $e(L_T) = \langle T\rangle_c$.

We view elements of ${\rm Aut}(\mu)$ as unitary operators on the Hilbert space $L_2^0(\mu, {\mathbb C})$ of all
measure classes of square summable complex functions with zero integral.
By \cite[Theorem 2]{Ha}, a generic element of ${\rm Aut}(\mu)$ does not have non-trivial eigenvectors in $L_2^0(\mu, {\mathbb C})$, and
we assume that our $T$ has this property. Now, by the spectral theorem \cite[Proposition 4.7.13]{Pe} applied to the
set of operators from $\langle T\rangle_c$, there is a Borel probability measure $\lambda$
on a standard Borel space and an isometry $U\colon L_2^0(\mu, {\mathbb C})\to L_2(\lambda, {\mathbb C})$ such that,
for each $S\in \langle T\rangle_c$, $USU^*$ is equal to the multiplication by an element of $L_0(\lambda, {\mathbb T})$.
We identify $USU^*$ with that element.  Note that since $T$ does not have non-trivial eigenvectors, $\lambda$ is atomless and, therefore,
can be assumed to be equal to Lebesgue measure. We see that the function
\[
\langle T\rangle_c\ni S\to USU^*\in L_0(\lambda, {\mathbb T})
\]
is a group isomorphism between $\langle T\rangle_c$ and its image and it is also a homeomorphism between
$\langle T\rangle_c$ and its image, where $\langle T\rangle_c$ is taken with the topology inherited from ${\rm Aut}(\mu)$ while its image is
taken with the
topology of convergence in $\lambda$. Since $\langle T\rangle_c$ is a Polish group, being a closed subgroup of ${\rm Aut}(\mu)$,
its isomorphic and homeomorphic image is closed.
It follows that
$\langle T\rangle_c$
can be identified with a closed subgroup of $L_0({\lambda}, {\mathbb T})$.

By the above identification, we have a continuous surjective homomorphism
\begin{equation}\label{E:incl}
e\upharpoonright L_T\colon L_T \to \langle T\rangle_c<L_0(\lambda, {\mathbb T}).
\end{equation}
For every $X\in L_T$, the function mapping $t\in {\mathbb R}$ to $e(X(t))$ is a one-parameter subgroup of
$\langle T\rangle_c$ and, therefore, by inclusion \eqref{E:incl}, of $L_0(\lambda, {\mathbb T})$.
By Lemma~\ref{L:lpa}, there exists a unique element $f(X)$ of $L_0(\lambda, {\mathbb R})$ such that $\exp(tf(X)) = e(X(t))$ for
all $t\in {\mathbb R}$. It is now easy to check that $f$ gives a Borel, and so continuous, linear function $f\colon L_T\to L_0(\lambda, {\mathbb R})$ such
that $\exp\circ f = e\upharpoonright L_T$. Now the closure in $L_0(\lambda, {\mathbb R})$ of $f(L_T)$ is a closed
linear subspace of $L_0(\lambda, {\mathbb R})$ that is as required by the conclusion of the theorem.
\end{proof}

To obtain Corollary~\ref{C:pro}, we need the following lemma. In the proof of the lemma, we will use the following result of
Ageev \cite[the proof of (1) on p. 217]{Ag}:

\noindent given a natural number $n\geq 1$, the projection on the first coordinate of a relatively open, non-empty
subset of $\{ (T,S) \in {\rm Aut}(\mu)^2\colon S^n = {\rm Id}\hbox{ and } ST=TS\}$ is non-meager.

\begin{lemma}\label{L:tor}
There is a comeager set of $T\in {\rm Aut}(\mu)$ with $\langle T\rangle_c$
containing a dense torsion subgroup.
\end{lemma}

\begin{proof} Fix a metric $d$ on ${\rm Aut}(\mu)$ compatible with the topology. Fix $\epsilon>0$. We claim that the set
\begin{equation}\notag
A_\epsilon = \{ T\in {\rm Aut}(\mu)\colon \exists S\; S\hbox{ has finite order}, ST=TS, \hbox{ and } d(S,T)<\epsilon\}
\end{equation}
is comeager. Let $S_0$ have finite order, say $n\in {\mathbb N}$. Then the set
\[
\{ (T,S)\in {\rm Aut}(\mu)^2\colon S^n = {\rm Id}, ST=TS,  d(T,S_0)<\epsilon, \hbox{ and }d(T,S)<\epsilon\}
\]
is relatively open in the set
\[
\{ (T,S)\in {\rm Aut}(\mu)^2\colon S^n = {\rm Id}\hbox{ and } ST=TS\}
\]
and is non-empty, as it contains $(S_0, S_0)$. It follows by Ageev's result quoted above that its projection on
the first coordinate is non-meager and is included in the $\epsilon$ ball around $S_0$. So each ball around $S_0$
contains a non-meager set of $T$ from the set $A_\epsilon$.
Since the set of such $S_0$ (with varying $n$) is dense, it follows that $A_\epsilon$ is comeager.
As a consequence, we see that $\bigcap_n A_{1/n}$ is comeager. This intersection consists
of transformations $T$ for which there exists a sequence $(S_n)$ of finite order transformations with $S_nT=TS_n$
and $S_n\to T$ as $n\to\infty$. Since, by the result of Akcoglu, Chacon, and Schwartzbauer \cite[Theorem 3.3]{ACS} already mentioned 
in the proof of Theorem~\ref{T:expo}, for a comeager set of $T$, the centralizer of $T$ is equal to
the closed group generated by $T$, the conclusion follows.
\end{proof}

\begin{proof}[Proof of Corollary~\ref{C:pro}]
(i) This point is immediate from Theorem~\ref{T:expo}.

(ii) It suffices to show that there are groups $K_n< \langle T\rangle_c$, $n\in {\mathbb N}$, such that
each $K_n$ is isomorphic to ${\mathbb T}$ and $\bigcup_n K_n$ is dense in $\langle T\rangle_c$. To see sufficiency of this condition,
let $L_n$ be the group generated by
$K_1\cup\cdots \cup K_n$. Then $L_n$ is a finite dimensional torus, $L_n<L_{n+1}$ for each $n$, and $\bigcup_nL_n$ is dense
in $\langle T\rangle_c$.
Possibly going to a subsequence, we can find linear subspaces $V_1<V_2<\cdots$ of the subspace of $L_0(\lambda, {\mathbb R})$
that maps homomorphically onto $\langle T\rangle_c$ found in (i), with
$V_n$ having dimension $n$, and continuous homomorphisms $\phi_n\colon V_n\to L_n$ with discrete
${\rm ker}(\phi_n)$ and with $\phi_{n+1}\upharpoonright V_n=\phi_n$. Inductively, we pick a basis $e_1, e_2, \dots$ of the linear space
$\bigcup_nV_n$ so that for each $n$
\[
{\rm ker}(\phi_n)= {\mathbb Z}e_1 + \cdots + {\mathbb Z}e_n.
\]
Now there is an isomorphism $f_n\colon {\mathbb T}^n\to L_n$ such that
\[
\phi_n\circ g_n = f_n\circ \pi_n,
\]
where $\pi_n\colon {\mathbb R}^n \to {\mathbb T}^n$ is the exponential map and $g_n\colon {\mathbb R}^n\to {\mathbb R}e_1+\cdots + {\mathbb R}e_n$
is the canonical linear isomorphism. Then $f_{n+1}\upharpoonright {\mathbb T}^n = f_n$ if ${\mathbb T}^n$ is considered as a subgroup of
${\mathbb T}^{n+1}$ consisting of points with the last coordinate equal to $1$. It follows that the sequence $(f_n)$ induces
a continuous embedding from ${\mathbb T}_\infty$ to $\langle T\rangle_c$ that is as required.

It remains to find groups $K_n$. By Lemma~\ref{L:tor}, we can pick transformations $S_n$ of finite order such that $S_n\to T$ as $n\to\infty$,
$S_nT=TS_n$, and $S_n\not= {\rm Id}$. Using again \cite[Theorem 3.3]{ACS}, we can assume that $S_n\in \langle T\rangle_c$ for all $n$.
By Theorem~\ref{T:expo}, there are one-parameter subgroups $X_n$, $n\in {\mathbb N}$, of
$\langle T\rangle_c$ with $X_n(1) = S_n$. Let $K_n = {\rm image}(X_n)$.
\end{proof}

\smallskip

\noindent {\bf Acknowledgement.} I would like to thank Aaron Hill for bringing to my attention
papers \cite{Ag2} and \cite{SE}.

\end{document}